\documentclass[12pt,a4paper]{amsart}
\usepackage{amssymb,amsmath,amsthm, amscd, enumerate, mathrsfs}
\usepackage{graphicx, hhline}
\usepackage[all]{xy}
\usepackage[usenames]{color}
\usepackage{hyperref}
\hypersetup{colorlinks=true}
\usepackage{color}

\title{Finite generation of adjoint ring for log surfaces}
\author{Kenta Hashizume} 
\date{2015/11/30, version {0.16}}
\address{Department of Mathematics, Graduate School of Science, 
Kyoto University, Kyoto 606-8502, Japan}
\email{hkenta@math.kyoto-u.ac.jp}

\newtheorem{thm}{Theorem}[section]
\newtheorem{lem}[thm]{Lemma}
\newtheorem{lema}{Lemma A}
\newtheorem{lemb}{Lemma B}

\newtheorem{cor}[thm]{Corollary}

\theoremstyle{definition}
\newtheorem{defn}[thm]{Definition}
\newtheorem{defnb}[lemb]{Definition B}
\newtheorem{rem}[thm]{Remark}
\newtheorem{remb}[lemb]{Remark B}
\newtheorem*{ack}{Acknowledgments} 

\newtheorem{step}{Step}
\begin{document}

\maketitle 

\begin{abstract}
We prove the finite generation of the adjoint ring for $\mathbb{Q}$-factorial log surfaces over 
any algebraically closed field.
\end{abstract}

\tableofcontents 

\section{Introduction}\label{sec1}

In this paper, we prove

\begin{thm}[Main Theorem]\label{main-thm}
Let $\pi: X \to U$ be a proper  
morphism from a normal $\mathbb{Q}$-factorial surface 
to a variety over an algebraically 
closed field and let $\Delta^{\bullet}=(\Delta_{1}, \, \cdots, \Delta_{n})$ be 
an $n$-tuple of boundary $\mathbb{Q}$-divisors. 
Then the adjoint ring of $(\pi, \Delta^{\bullet})$
$$\mathcal{R}(\pi, \Delta^{\bullet})=
\underset{(m_{1}, \, \cdots, \,m_{n}) \in (\mathbb{Z}_{ \geq 0})^{n}}{\bigoplus}
\pi_{*} \mathcal{O}_{X}(\lfloor \sum_{i=1}^{n}m_{i}(K_{X}+\Delta_{i}) \rfloor)$$
is a finitely generated $\mathcal{O}_{U}$-algebra. 
\end{thm}

As a corollary, we have

\begin{cor}\label{cor1.2}
Let $\pi:X \to U$ be a proper morphism from a normal surface to a variety 
over an algebraically closed field and 
let $\Delta^{\bullet}=(\Delta_{1}, \, \cdots , \Delta_{n})$ 
be an $n$-tuple of $\mathbb{Q}$-divisors such that 
$(X, \Delta_{i})$ is log canonical for every 
$1 \leq i \leq n$. 
Then the adjoint ring 
$\mathcal{R}(\pi, \Delta^{\bullet})$ 
of $(\pi,\Delta^{\bullet})$ is a finitely generated $\mathcal{O}_{U}$-algebra. 
\end{cor}

Theorem \ref{main-thm} is a generalization of \cite[Corollary 1.5]{fujita} 
and \cite[Corollary 1.3]{fujinotanaka}. 
In higher dimension, the finite generation of the adjoint ring is known by Birkar, Cascini, Hacon and 
M\textsuperscript{c}Kernan \cite{bchm} when
$\pi:X \to U$ is a projective morphism of normal quasi-projective varieties over the complex number field 
and $\Delta^{\bullet}=(A+B_{1}, \, \cdots, A+B_{n})$, where $A\geq0$ is a general $\pi$-ample 
$\mathbb{Q}$-divisor and $B_{i}$ is an effective $\mathbb{Q}$-divisor such that 
$(X,\,A+B_{i})$ is divisorial log terminal for any $i$.
We emphasize that $(X, \Delta_i)$ is not necessarily log canonical in Theorem \ref{main-thm}.

Let us summarize the history of Theorem \ref{main-thm}. 
In \cite{fujita}, Takao Fujita established Theorem \ref{main-thm} under the assumption that 
$X$ is nonsingular, $U$ is a point, and $n=1$. 
More precisely, he proved that 
the positive part of the Zariski decomposition of 
$K_X+\Delta$ is semi-ample by using the notion of 
Sakai minimality
when $X$ is a nonsingular projective surface and 
$\Delta$ is a boundary $\mathbb Q$-divisor on $X$, that is, a 
$\mathbb Q$-divisor on $X$ whose 
coefficients are in $[0, 1]$. 
As an easy consequence, he 
obtained the above mentioned special case of Theorem \ref{main-thm}. 
In \cite{fujino}, Osamu Fujino established the minimal model program (MMP) and 
the abundance theorem for $\mathbb Q$-factorial 
log surfaces and log canonical surfaces over an 
algebraically closed field of characteristic zero in full generality. 
Theorem \ref{main-thm} with $n=1$ 
in characteristic zero follows immediately 
from the minimal model program and the abundance theorem 
for $\mathbb Q$-factorial log surfaces \cite{fujino}. 
We note that Fujino's approach based on \cite{fujino0} heavily 
depends on vanishing theorems of Kodaira type. 
Therefore, we can not directly apply his arguments in positive characteristic. 
Fortunately, in \cite{tanaka}, Hiromu Tanaka generalized the results in \cite{fujino} 
for $\mathbb Q$-factorial log surfaces and log canonical surfaces in positive characteristic. 
Consequently, we know that 
Theorem \ref{main-thm} holds true over any algebraically closed field when $n=1$.

As we mentioned above, 
we are now able to use 
the minimal model program and the abundance theorem 
for $\mathbb{Q}$-factorial log surfaces, 
which need not be log canonical.
In this paper, we prove Theorem \ref{main-thm} in full generality by 
using the minimal model program and the abundance theorem for 
$\mathbb Q$-factorial log surfaces 
established in \cite{fujino} and \cite{tanaka} (see also \cite{fujinotanaka}). 
We will use Shokurov's ideas in \cite{shokurov} in order to reduce Theorem \ref{main-thm} 
to the case when $K_X+\Delta_i$ is semi-ample over $U$ for every $i$. 
Note that we have to use the minimal model program and the 
abundance theorem for $\mathbb R$-divisors to carry out Shokurov's ideas. 

The contents of this paper are as follows. 
In Section \ref{sec2}, we recall the definition of log surfaces 
and collect some other basic definitions and notations.
In Section \ref{sec3}, which is the main part 
of this paper, we discuss a certain cone decomposition of  
the cone of pseudo-effective divisors.
See Lemma \ref{lem3.4} for details.
In Section \ref{sec4}, we reduce the proof of Theorem \ref{main-thm}  
to the case where $K_{X}+\Delta_{i}$ is semi-ample for every $i$.
To carry it out, 
we use some basic properties of 
graded rings and rational polytopes, which are rather technical.   
In Section \ref{sec5}, we complete the proof of 
Theorem \ref{main-thm} and Corollary \ref{cor1.2}.
In Section \ref{sec6}, which is Appendix, 
we collect some basic results on graded rings and rational polytopes 
used in Section \ref{sec4} for the reader's convenience.

Throughout this paper, we work over an 
algebraically closed field of any characteristic. 

\begin{ack}
The author would like to thank his supervisor Osamu Fujino for 
various suggestions 
and warm encouragement. 
He is grateful to the referee for many valuable comments. 
He also thanks 
his colleagues for discussions. 
\end{ack}

\section{Notations and definitions}\label{sec2}

In this section we collect some notations and definitions.
Let $k$ be an algebraically closed field. A {\it variety} is a 
separated integral scheme of finite type over $k$. 
Let $X$ be a normal variety and 
let $\pi :X \rightarrow U$ be a morphism from $X$ to a variety 
$U$. 

\begin{itemize}
\item[(1)]
${\rm WDiv}_{\mathbb{R}}(X)$ is 
the $\mathbb{R}$-vector space with canonical basis given 
by the prime divisors of $X$. 

\item[(2)]
An $\mathbb{R}$-divisor $D$ on $X$ is $\mathbb{Q}${\it -Cartier} 
(resp.~$\mathbb{R}${\it -Cartier})  if $D$ is a 
$\mathbb{Q}$-linear  (resp.~an $\mathbb{R}$-linear) combination of Cartier divisors.

\item[(3)]
$X$ is $\mathbb{Q}${\it -factorial} if every Weil divisor is $\mathbb{Q}$-Cartier. 

\item[(4)]
Two $\mathbb{R}$-divisors $D$ and $D'$ 
on $X$ are $\mathbb{Q}${\it -linearly equivalent} 
(resp.~$\mathbb{R}${\it -linearly equivalent}), 
denoted by $D \sim _{\mathbb{Q}} D'$ (resp.~$D \sim _{\mathbb{R}} D'$), 
 if $D-D'$ is a $\mathbb{Q}$-linear  (resp.~an $\mathbb{R}$-linear) 
combination of principal divisors.

\item[(5)]
Two $\mathbb{R}$-divisors $D$ and $D'$ on $X$ are 
$\mathbb{Q}${\it -linearly equivalent over} $U$ 
(resp.~$\mathbb{R}${\it -linearly equivalent over} $U$), 
denoted by $D \sim_{\mathbb{Q}, \,U} D'$ 
(resp.~$D \sim_{\mathbb{R}, \,U} D'$), if there exists a 
$\mathbb{Q}$-Cartier (resp.~an $\mathbb{R}$-Cartier) divisor $E$ on $U$ 
such that $D-D' \sim_{\mathbb{Q}} \pi^{*}E$ 
(resp.~$D-D' \sim_{\mathbb{R}} \pi^{*}E$). 

\item[(6)]
Two $\mathbb{R}$-divisors $D$ and $D'$ on $X$ are {\it numerically equivalent over} $U$ 
(or $\pi${\it -numerically equivalent}) if $D-D'$ is 
$\mathbb{R}$-Cartier and $(D-D' )\cdot C=0$ 
for every proper curve $C$ on 
$X$ contained in a fiber of $\pi$. 

\item[(7)]
An $\mathbb{R}$-divisor $D$ on $X$ 
is {\it pseudo-effective over} $U$ (or $\pi${\it -pseudo-effective}) 
if $D$ is $\pi$-numerically equivalent to the limit of effective
$\mathbb{R}$-divisors modulo numerically equivalence over $U$. 

\item[(8)]
An $\mathbb{R}$-Cartier divisor $D$ on $X$ is {\it nef over} $U$ (or $\pi${\it -nef}) if 
$(D \cdot C) \geq 0$ for every proper curve 
$C$ on $X$ contained in a fiber of $\pi$.

\item[(9)]
An $\mathbb{R}$-divisor $D$ is 
{\it semi-ample over} $U$ (or $\pi${\it -semi-ample}) 
if $D$ is an $\mathbb{R}_{\geq 0}$-linear 
combination of semi-ample Cartier divisors over $U$, or equivalently, 
there exists a morphism $f: X \rightarrow Y$ 
to a variety over $U$ such that $D$ is $\mathbb{R}$-linearly 
equivalent to the pullback of an ample $\mathbb{R}$-divisor over $U$. 

\item[(10)]
For a real number $\alpha$, 
its {\it round down} is the largest integer which is not greater than $\alpha$. 
It is denoted by $\lfloor \alpha \rfloor$. 
If $D= \sum \alpha_{i} D_{i}$ is 
an $\mathbb{R}$-divisor and the $D_{i}$ are distinct prime divisors, 
then the {\it round down} of $D$, 
denoted by $\lfloor D \rfloor$, is $\sum \lfloor \alpha _{i} \rfloor D_{i}$. 

\item[(11)] An 
$\mathbb{R}$-divisor $D$ 
on $X$ is a {\it boundary} $\mathbb{R}${\it -divisor} if $D$ is effective and
whose coefficients are not greater than one.

\item[(12)]
Let $K_{X}$ be the canonical 
divisor on $X$ and let $V$ be a finite dimensional affine subspace of 
${\rm WDiv}_{\mathbb{R}}(X)$. 
Then we define $\mathcal{E}_{U}(V)$ as 
$$\mathcal{E}_{U}(V)=\{ \Delta \in V \mid K_{X}+ \Delta \; {\rm is} \; 
\pi \mathchar`-{\rm pseudo}\mathchar`-{\rm effective} \}.$$ 

\item[(13)]
A pair $(X, \Delta)$, where $X$ is a normal variety and $\Delta$ is an effective 
$\mathbb{R}$-divisor on $X$, is 
said to be {\it log canonical} 
if $K_{X}+ \Delta$ is $\mathbb{R}$-Cartier and for any proper birational morphism
$f:Y \rightarrow X$ from a normal 
variety $Y$, every coefficient of $K_{Y}-f^{*}(K_{X}+ \Delta)$, where $K_{Y}$ is the canonical divisor on $Y$ such that 
$f_{*}K_{Y}=K_{X}$, is greater than or equal to $-1$.

\item[(14)]
Let $D^{\bullet}=(D_{1}, \, \cdots , D_{n})$ be an 
$n$-tuple of $\mathbb{Q}$-divisors on $X$. 
Then we define the sheaf of $\mathcal{O}_{U}$-algebra $\mathcal{R}(\pi, D^{\bullet})$ as
$$\mathcal{R}(\pi, D^{\bullet})
=\underset{(m_{1}, \, \cdots ,\, m_{n}) \in (\mathbb{Z}_{ \geq 0})^{n}}{\bigoplus}
\pi_{*} \mathcal{O}_{X}(\lfloor \sum_{i=1}^{n}m_{i}(K_{X}+D_{i}) \rfloor)$$
and call it {\it adjoint ring} of $(\pi:X\to U, D^\bullet$). 
We note that 
$\mathcal{R}(\pi, D^{\bullet})$ 
is a {\it finitely generated} $\mathcal{O}_{U}${\it -algebra} if 
there is an affine open
covering $\{ V_{i}={\rm Spec}\,A_{i} \}_{i \in I}$ of $U$ such that 
$\mathcal{R}(\pi, D^{\bullet}) \! \mid _{V_{i}}$ is 
the sheaf associated to a finitely generated $A_{i}$-algebra 
for every $i \in I$.
\end{itemize}

Next, we recall the  definition of log surfaces. 

\begin{defn}[Log surfaces]\label{defn2.1}
Let $X$ be a normal surface and let 
$\Delta$ be a boundary $\mathbb{R}$-divisor on $X$ 
such that $K_{X}+\Delta$ is $\mathbb{R}$-Cartier. 
Then the pair $(X, \Delta)$ is called a {\it log surface}.
\end{defn}

Finally, we recall the definition of weak log canonical models and minimal models of log surfaces.  

\begin{defn}[cf.~{Definition 3.6.1 \cite{bchm}}, Definition 3.6.7 \cite{bchm}]\label{def3.3}
Let $\pi : X \to U$ and $\pi ' :Y \to U$ be projective 
morphisms from a normal surface to a variety. 
Let $f:X \to Y$ be a birational morphism of normal 
surfaces over $U$ and let $D$ be an 
$\mathbb{R}$-Cartier divisor on $X$ such that $f_{*}D$ is also $\mathbb{R}$-Cartier. 
Then $f$ is $D${\it -non-positive} (resp.~$D${\it -negative}) if $E=D-f^{*}f_{*}D$ is an 
effective $f$-exceptional divisor 
(resp.~an effective $f$-exceptional divisor and the support of $E$ contains 
supports of all $f$-exceptional divisors). 
Let $(X,\,\Delta)$ be a log surface and $f:X \to Y$ be a birational morphism over $U$. 
Then $f$ is a {\it weak log canonical model} (resp.~{\it minimal model}) of $(X, \Delta)$ over $U$ if 
$f$ is $(K_{X}+\Delta)$-non-positive (resp.~$(K_{X}+\Delta)${\it -negative}) and 
$K_{Y}+f_{*}\Delta$ is nef over $U$. 

\end{defn}

\begin{rem}\label{rmk2.3}
Let $\pi : X \to U$, $\pi ' :Y \to U$ and $f:X\to Y$ be as above and let $V$ be a finite dimensional 
affine subspace in ${\rm WDiv}_{\mathbb{R}}(X)$. 
Then we can easily check that the set 
$$\{ \Delta \in V \mid f\;{\rm is\; a\; weak\; log\; canonical\; model\; of\; } 
(X, \Delta)\;{\rm over}\; U\}$$
is a closed convex subset in $V$.
\end{rem}

\section{Cone decomposition}\label{sec3}
In this section, we discuss a certain cone decomposition of  
the cone of pseudo-effective divisors by using
the minimal model program and the 
abundance theorem for $\mathbb Q$-factorial log surfaces. 
This cone decomposition will play a crucial role in Section \ref{sec4}. 
For the definition of rational polytopes and their faces, 
see Definition B \ref{defb6.1}.

\begin{lem}[cf.~\cite{shokurov}]\label{lem3.4}
Let $\pi : X \to U$ be a projective morphism 
from a normal $\mathbb{Q}$-factorial surface onto a 
quasi-projective variety. Let $V$ be a finite dimensional affine subspace of 
${\rm WDiv}_{\mathbb{R}}(X)$, which is defined over $\mathbb Q$, 
and let $V'$ be the set of  the
boundary $\mathbb{R}$-divisors on 
$X$ contained in $V$. Let $\mathcal{C}$ be a rational polytope in $V'$. 
Then there are finitely many proper birational 
morphisms $f_{i}: X \to Y_{i}$ over $U$ and finitely many 
rational polytopes $W_{i}$ such that 
$\mathcal{C} \cap \mathcal{E}_{U}(V)= 
\cup_{i}W_{i}$ and if $\Delta \in W_{i}$, 
then $f_{i}$ is a weak log canonical model of $(X, \Delta)$ over $U$. In particular, 
$\mathcal{C} \cap \mathcal{E}_{U}(V)$ is also a rational polytope.
\end{lem}

\begin{proof}
Without loss of generality, we may assume that $\mathcal{C}$ spans $V$ 
by replacing $V$ with the span of $\mathcal{C}$. 
We proceed by induction on the dimension of $\mathcal{C}$.

If ${\rm dim} \,\mathcal{C}=0$, then 
we may assume that $\{ D\} = \mathcal{C} \cap \mathcal{E}_{U}(V)$. 
Then, by \cite[Theorem 1.2]{fujinotanaka}, 
there exists a minimal model $f: X \to Y$ of $(X, D)$ over $U$. 
By Definition \ref{def3.3}, a minimal model of $(X, D)$ over $U$ is a weak log 
canonical model of $(X, D)$ over $U$. 
Thus $f$ and $W=\{ D\}$ satisfy the conditions of the lemma. 
So we may assume that ${\rm dim}\,\mathcal{C}>0$. 

We show the assertion in  the lemma assuming that there is an 
$\mathbb{R}$-divisor $\Delta_{0} \in \mathcal{C} \cap \mathcal{E}_{U}(V)$ 
such that $K_{X}+\Delta_{0} \sim_{\mathbb{R}, \,U} 0$. 
We first show that 
there is a $\mathbb{Q}$-divisor 
$\Delta'$ in $\mathcal{C} \cap \mathcal{E}_{U}(V)$ 
such that $K_{X}+\Delta' \sim_{\mathbb{Q},\,U}0$.
Indeed, we may write 
$K_{X}+\Delta_{0}=\sum_{i=1}^{k}\alpha_{i}(f_{i})
+\sum_{j=1}^{l}\beta_{j}\pi^{*}F_{j}$ for some $\alpha_{i},\beta_{j} \in \mathbb{R}$, 
principal divisors $(f_{i})$ and Cartier divisors $F_{j}$ on $U$. 
Let $T$ be the finite dimensional $\mathbb{R}$-vector 
space in ${\rm WDiv}_{\mathbb{R}}(X)$ spanned by $(f_{i})$ and $\pi^{*}F_{j}$. 
Then $T$ is defined over $\mathbb{Q}$.
Therefore, the set
$$\{ \Delta \in \mathcal{C} \mid K_{X}+\Delta \in T \}$$
is a rational polytope and contains $\Delta_{0}$. 
In particular, this set is a non-empty rational polytope. 
Then a rational point $\Delta'$ in this set 
satisfies $K_{X}+\Delta' \sim_{\mathbb{R},\,U}0$, and so $\Delta' \in \mathcal{E}_{U}(V)$. 
Since $K_{X}+\Delta'$ is a $\mathbb{Q}$-divisor, 
we obtain $K_{X}+\Delta' \sim_{\mathbb{Q},\,U}0$.
By replacing $\Delta_{0}$ with $\Delta'$, we 
may assume that $\Delta_{0}$ is a $\mathbb{Q}$-divisor and 
$K_{X}+\Delta_{0} \sim_{\mathbb{Q}, \,U} 0$. Pick $D \in \mathcal{C}$ 
with $D \neq \Delta_{0}$. 
Then there is a divisor $D'$ on the boundary of $\mathcal{C}$ such that 
$D-\Delta_{0}= \lambda(D'-\Delta_{0})$ for some $0< \lambda \leq 1$. 
Then 
$$K_{X}+D = \lambda (K_{X}+D')+(1-\lambda)(K_{X}+\Delta_{0})
\sim_{\mathbb{R}, \,U} \lambda (K_{X}+D').$$ 
In particular, 
$K_{X}+D$ is pseudo-effective over $U$ if and only if 
$K_{X}+D'$ is pseudo-effective over $U$. 
Moreover, the pairs $(X, D)$ and $(X, D')$ have the same weak log canonical model over $U$ 
by \cite[Lemma 3.6.9]{bchm}. 
Let $\partial \mathcal{C}$ be the boundary of $\mathcal{C}$. 
Since $\partial \mathcal{C}$ 
consists of finitely many rational 
polytopes, there are finitely many 
proper birational morphisms 
$f_{i}: X \to Y_{i}$ over $U$ and finitely many rational polytopes $W'_{i}$ such that 
$\partial \mathcal{C} \cap \mathcal{E}_{U}(V)
= \cup_{i}W'_{i}$ and if $D \in W'_{i}$, 
then $f_{i}$ is a weak log canonical model of $(X, D)$ over $U$. 
Let $W_{i}$ be the cone spanned by $\Delta_{0}$ and $W'_{i}$. 
Then $f_{i}$ and $W_{i}$ satisfy the conditions of the lemma. 
So we are done.

We now prove the general case. 
Since $V'$ is compact and 
$\mathcal{C} \cap \mathcal{E}_{U}(V)$ is closed in $V'$, it is sufficient 
to prove the lemma for every $\Delta_{0} \in \mathcal{C} \cap \mathcal{E}_{U}(V)$ 
and a 
sufficiently small neighborhood 
$\mathcal{C}_{0}$ of $\Delta_{0}$ in $\mathcal{C}$, which is also 
a rational polytope. 
By \cite[Theorem 1.2]{fujinotanaka}, 
there exists a minimal model $f_{0}: X\to Y_{0}$ of $(X, \Delta_{0})$ over $U$. 
Since a minimal model of 
$(X, \Delta_{0})$ over $U$ is $(K_{X}+\Delta_{0})$-negative, possibly shrinking 
$\mathcal{C}_{0}$, we may 
assume that for any $\Delta \in \mathcal{C}_{0}$, $f_{0}$ is $(K_{X}+\Delta)$-non-positive. 
We put $W=f_{0*}(V)$ and $\mathcal{C}'$
=$f_{0*}(\mathcal{C}_{0})$. 
Then $\mathcal{C}' \subset W$ 
is a rational polytope containing 
$f_{0*}\Delta_{0}$ and ${\rm dim}\,\mathcal{C}' \leq {\rm dim} \,\mathcal{C}$. 

By 
\cite[Theorem 8.1]{fujino} and \cite[Theorem 6.7]{tanaka}, $K_{Y_0}
+f_{0*}\Delta_{0}$ is semi-ample over $U$. 
Then there exists a 
projective morphism 
$\phi_{0}: Y_{0} \to Z_{0}$ onto a quasi-projective variety $Z_{0}$ over 
$U$ and an ample $\mathbb{R}$-divisor  $A$ over $U$ such that 
$K_{Y_{0}}+f_{0*}\Delta_{0} \sim_{\mathbb{R}} \phi_{0}^{*}A$. 
Since $K_{Y_{0}}+f_{0*}\Delta_{0} \sim_{\mathbb{R},\, Z_{0}}0$, 
there are finitely many 
proper birational morphisms $h_{i}: Y_{0} \to Y_{i}$ over $Z_{0}$ and 
finitely many rational 
polytopes $W_{i}$ such 
that $\mathcal{C}' \cap \mathcal{E}_{Z_{0}}(W)= \cup_{i}W_{i}$ and 
if $D \in W_{i}$, then $h_{i}$ is a weak log canonical model of $(Y_{0}, D)$ over $Z_{0}$. 
Possibly shrinking $\mathcal{C}_{0}$, 
we may assume that $f_{0*}\Delta_{0} \in W_{i}$ for any $i$. 
Let $\phi_{i}: Y_{i} \to Z_{0}$ be the induced morphism. 
Pick a vertex $\Delta$ of $W_{i}$. 
Then $K_{Y_{i}}+h_{i*}\Delta$ is nef over $Z_{0}$ and by \cite[Theorem 8.1]{fujino} and 
\cite[Theorem 6.7]{tanaka}, $K_{Y_{i}}+h_{i*}\Delta$ is semi-ample over $Z_{0}$. 
Therefore $K_{Y_{i}}+h_{i*}\Delta+n\phi_{i}^{*}A$ 
is semi-ample over $U$ for a large integer $n$. 
In particular $K_{Y_{i}}+h_{i*}\Delta+n\phi_{i}^{*}A$ is nef over $U$.
If $0<\epsilon <1/(n+1)$, then
$$h_{i*}(K_{Y_{0}}+\epsilon \Delta+(1-\epsilon)f_{0*}\Delta_{0})
\sim_{\mathbb{R}}\epsilon(K_{Y_{i}}+h_{i*}\Delta)+(1-\epsilon)\phi_{i}^{*}A$$
is nef over $U$. 
Considering all vertices of all 
$W_{i}$, we may find a sufficiently 
small neighborhood $\mathcal{C}''$ of $f_{0*}\Delta_{0}$ in $\mathcal{C}'$, which is a rational 
polytope, such that if 
$D' \in \mathcal{C}'' \cap W_{i}$, then $h_{i}$ 
is a weak log canonical model of $(Y_{0}, D')$ over $U$. 
Then $K_{Y_{i}}+h_{i*}D'$ is nef over $U$. 
In particular,  
$K_{Y_{i}}+h_{i*}D'$ is pseudo-effective over $U$. 
Since $h_{i}$ is $(K_{Y_{0}}+D')$-non-positive, 
$K_{Y_{0}}+D'$ is also pseudo-effective over $U$. 
On the other hand, a pseudo-effective divisor over 
$U$ is also pseudo-effective over $Z_{0}$. 
Therefore, 
$\cup_{i}(\mathcal{C}'' \cap W_{i} )=
\mathcal{C}'' \cap \mathcal{E}_{Z_{0}}(W)=\mathcal{C}'' \cap \mathcal{E}_{U}(W)$. 
Set $\widetilde{\mathcal{C}}=(f_{0*})^{-1} (\mathcal{C}'') \cap \mathcal{C}_{0}$ and 
$\widetilde{W_{i}}=(f_{0*})^{-1}(W_{i} \cap \mathcal{C}'') \cap \mathcal{C}_{0}$. 
Then $\widetilde{\mathcal{C}}$ is a 
neighborhood of $\Delta_{0}$ in $\mathcal{C}_{0}$ and a rational polytope. 
If $\Delta \in \widetilde{\mathcal{C}} \cap \mathcal{E}_{U}(V)$, then 
$f_{0*}\Delta \in \mathcal{C}'' \cap \mathcal{E}_{U}(W)$. 
Therefore we have $\Delta \in \cup_{i}\widetilde{W_{i}}$. 
On the other hand, if $\Delta \in \cup_{i}\widetilde{W_{i}}$, then 
$f_{0*}\Delta \in \mathcal{C}'' \cap \mathcal{E}_{U}(W)$. 
In particular, $K_{Y_{0}}+f_{0*}\Delta$ is pseudo-effective over $U$.
Since $\Delta \in \mathcal{C}_{0}$, $f_{0}$ is $(K_{X}+\Delta)$-non-positive. 
Therefore $\Delta \in  \mathcal{E}_{U}(V)$. 
Thus, we see that $\widetilde{\mathcal{C}} \cap \mathcal{E}_{U}(V)= \cup_{i}\widetilde{W_{i}}$. 
We can also check that if $\Delta \in \widetilde{W_{i}}$, then $h_{i} \circ f_{0}$ is a weak log canonical 
model of $(X,\Delta)$ over $U$.
So we are done. 
\end{proof}

\section{Reduction to the special case}\label{sec4}

In this section, we reduce Theorem \ref{main-thm} to the case 
where $K_{X}+\Delta_{i}$ is semi-ample over $U$ 
for every $1 \leq i \leq n$.

We put $D_{i}=K_{X}+\Delta_{i}$ for every $i$. 
We note that $\pi:X\to U$ is projective since $X$ is a 
$\mathbb Q$-factorial surface 
(see \cite[Lemma 2.2]{fujino}). 
By taking the Stein factorization, 
we may assume that $\pi_{*}\mathcal{O}_{X}=\mathcal{O}_{U}$. 
Furthermore, we may also assume that $U$ is an affine variety 
by the definition of finitely generated $\mathcal{O}_{U}$-algebras. 
Set $A=H^{0}(X, \mathcal{O}_{X}) =H^{0}(U, \mathcal{O}_{U})$.
Then it is sufficient to prove that
$$\mathcal{R}(\pi, \Delta^{\bullet})
=\underset{(m_{1}, \, \cdots , \, m_{n}) \in (\mathbb{Z}_{ \geq 0})^{n}}{\bigoplus}
H^{0}(X, \mathcal{O}_{X} (\lfloor \sum_{i=1}^{n}m_{i}D_{i} \rfloor ))$$ 
is a finitely generated $A$-algebra. 
Let $V$ (resp.~$\mathcal{C}$) be the affine subspace (resp.~convex hull) in 
${\rm WDiv}_{\mathbb{R}}(X)$ spanned by $\Delta_{1}, \, \cdots , \Delta_{n}$.

\begin{lem}\label{lem4.1a}
To prove Theorem \ref{main-thm}, we may assume that ${\rm dim}\,V={\rm dim}\, \mathcal{C}=n-1$, 
\end{lem}

\begin{proof}
We prove it with several steps.

\begin{step}\label{step1}
In this step, we reduce Theorem \ref{main-thm} to the case that 
$\Delta_{i}$ is a vertex of 
$\mathcal{C}$ for every $i$ and $\Delta_{i} \neq \Delta_{j}$ for any $i \neq j$. 

Suppose that there is an index $i$ such that $\Delta_{i}$ is not a vertex of $\mathcal{C}$ or 
there are two indices $i$ and $j$ such that $i\neq  j$ and $\Delta_{i}=\Delta_{j}$.   
By changing indices, 
we may write $\Delta_{n}= \sum_{i=1}^{k}(a_{i}/q)\Delta_{i}$ for some $1\leq k\leq n-1$,
$a_{i} \in \mathbb{Z}_{> 0}$ and $q \in \mathbb{Z}_{>0}$ such that  $\sum_{i=1}^{k}(a_{i}/q)=1$. 
Then we have $qD_{n}=\sum_{i=1}^{k}a_{i}D_{i}$.
By Lemma A \ref{cora2}, it is sufficient to prove the finite generation of 
$$\underset{(m_{1}, \, \cdots, \, m_{n}) \in (\mathbb{Z}_{\geq 0})^{n}}{\bigoplus}
H^{0}(X, \mathcal{O}_{X} (\lfloor( \sum_{i=1}^{k}m_{i}a_{i}D_{i}
+\sum_{i=k+1}^{n-1}m_{i}D_{i}
+m_{n}qD_{n} )\rfloor))$$
as an $A$-algebra. 
By Lemma A \ref{lema3}, it is sufficient to prove that
$$\underset{(m_{1},\,\cdots \overset{\overset{\overset{j}{\vee}}{}}{} \cdots ,\, m_{n}) 
\in (\mathbb{Z}_{\geq 0})^{n-1}}{\bigoplus}
H^{0}(X, \mathcal{O}_{X} (\lfloor (\sum_{\substack{{i=1} \\ 
{i\neq j}}}^{k}m_{i}a_{i}D_{i} +\sum_{i=k+1}^{n-1}m_{i}D_{i}
+m_{n}qD_{n})\rfloor))
$$
is a finitely generated $A$-algebra for every $1 \leq j \leq k$ and moreover we can reduce it to the finite generation of 
\begin{equation*}
\begin{split}
&\underset{(m_{1}, \, \cdots, \, m_{j-1},\,m_{j+1},\, \cdots ,\,m_{n}) 
\in (\mathbb{Z}_{\geq 0})^{n-1}}{\bigoplus}
H^{0}(X, \mathcal{O}_{X} (\lfloor (\sum_{i=1,\,i \neq j}^{n-1}m_{i}D_{i} +m_{n}D_{n})\rfloor))
\end{split}
\end{equation*}
as an $A$-algebra for any $1\leq j \leq k$ by using Lemma A \ref{cora2} again. By repeating this discussion, we may assume that $\Delta_{i}$ is a vertex of 
$\mathcal{C}$ for any $i$ and $\Delta_{i} \neq \Delta_{j}$ for any $i \neq j$.
\end{step}

\begin{step}\label{step2}
Assume that $\Delta_{i}$ is a vertex of 
$\mathcal{C}$ for every $i$ and $\Delta_{i} \neq \Delta_{j}$ for any $i \neq j$. 
In addition, suppose that ${\rm dim}\,V+1<n$. 
The goal of this step is to decrease $n$, the number of the boundary $\mathbb{Q}$-divisors, 
under the above assumption.

Let $\mathcal{C}_{j}$ be the convex hull 
spanned by $\Delta_{1}, \, \cdots,\Delta_{j-1},\Delta_{j+1}, \, 
\cdots, \Delta_{n}$ for every $j$. We can pick a $\mathbb{Q}$-divisor 
$\Delta_{n+1}$ such that 
$\Delta_{n+1} \in \mathcal{C}_{j}$ for any $1 \leq j \leq n$ by 
Lemma B \ref{lemb1}. 
Set $D_{n+1}=K_{X}+\Delta_{n+1}$. 
If 
$$\underset{(m_{1}, \, \cdots , \, m_{n+1}) \in (\mathbb{Z}_{ \geq 0})^{n+1}}{\bigoplus}
H^{0}(X, \mathcal{O}_{X} (\lfloor \sum_{i=1}^{n+1}m_{i}D_{i} \rfloor ))$$
is a finitely generated $A$-algebra, 
then 
it is obvious that 
$\mathcal{R}(\pi, \Delta^{\bullet})$ is also 
a finitely generated $A$-algebra. 
Moreover, by Lemma A \ref{cora2} and Lemma A \ref{lema3}, it is sufficient to prove that
$$\underset{(m_{1}, \, \cdots, \, m_{j-1},\,m_{j+1}, \,
\cdots ,\,m_{n+1}) \in (\mathbb{Z}_{\geq 0})^{n}}{\bigoplus}
H^{0}(X, \mathcal{O}_{X} (\lfloor \sum_{i=1,\, i \neq j}^{n+1}m_{i}D_{i} \rfloor))$$
is a finitely generated $A$-algebra for any $1\leq j \leq n$. 
Since $\Delta_{n+1} \in \mathcal{C}_{j}$, 
by using Lemma A \ref{cora2} and Lemma A \ref{lema3} again, we can reduce it to the finite generation of
$$\underset{(m_{1},\,\cdots \overset{\overset{\overset{j}{\vee}}{}}{} \cdots 
\overset{\overset{\overset{j'}{\vee}}{}}{} \cdots ,\, m_{n+1}) 
\in (\mathbb{Z}_{\geq 0})^{n-1}}{\bigoplus}
H^{0}(X, \mathcal{O}_{X} (\lfloor \sum_{i=1,\, i \neq j,\,j'}^{n+1}m_{i}D_{i} \rfloor))$$
as an $A$-algebra for any $1\leq j, j' \leq n$ with $j\ne j'$. 
\end{step}
\begin{step}
By repeating the discussion of Step \ref{step1} and Step \ref{step2}, 
we may assume that ${\rm dim}\,V+1=n$. 
So we are done.
\end{step}
\end{proof}

By Lemma \ref{lem4.1a}, we may assume that  ${\rm dim}\,V={\rm dim}\,\mathcal{C}=n-1$,  
or equivalently, any point of $V$ is represented uniquely by the
$\mathbb{R}$-linear combination 
of $\Delta_{1}, \, \cdots, \Delta_{n}$, where the sum of coefficients is equal to one.
Next we prove the following lemma. 
For simplex coverings, see Remark B \ref{remb4}. 

\begin{lem}\label{lem4.1}
Suppose that there is a finite $(n-1)$-dimensional rational simplex covering 
$\{ \Sigma_{\lambda}\}_{\lambda \in \Lambda}$ of $\mathcal{C}$ such that 

\begin{itemize}
\item[(i)]
$\mathcal{C}= \cup_{\lambda \in \Lambda} \Sigma_{\lambda}$, and

\item[(ii)] 
if $\Psi_{\lambda 1}, \, \cdots , \Psi_{\lambda n}$ 
are the vertices of $\Sigma_{\lambda}$, then 
$$R^{\lambda}=\underset{(m_{1}, \, \cdots,\; m_{n}) \in (\mathbb{Z}_{\geq 0})^{n}}
{\bigoplus} 
H^{0}(X, \mathcal{O}_{X}(\lfloor \sum_{i=1}^{n} m_{i}(K_{X}+\Psi_{\lambda i}) \rfloor))$$
is a finitely generated $A$-algebra for any $\lambda$.
\end{itemize}
Then $\mathcal{R}(\pi, \Delta^{\bullet})$ is also a finitely generated $A$-algebra.
\end{lem}

\begin{proof}
By hypothesis, for any $\lambda \in \Lambda$ and any $1\leq i \leq n$, we may write 
$\Delta_{i}=\sum_{j=1}^{n}(a_{\lambda ij}/p)\Psi_{\lambda j}$ and
$\Psi_{\lambda i}= \sum_{j=1}^{n}(b_{\lambda ij}/q)\Delta_{j}$, 
where $a_{\lambda ij} \in \mathbb{Z}$, 
$b_{\lambda ij} \in \mathbb{Z}_{\geq 0}$, $p,\;q \in \mathbb{Z}_{> 0}$ and 
$\sum_{j=1}^{n}(a_{\lambda ij}/p)=\sum_{j=1}^{n}(b_{\lambda ij}/q)=1$. 
Then
$$\Delta_{i}=\sum_{j=1}^{n}\frac{a_{\lambda ij}}{p}\Psi_{\lambda j}=
\sum_{j=1}^{n}\frac{a_{\lambda ij}}{p}\sum_{k=1}^{n}\frac{b_{\lambda jk}}{q}\Delta_{k}=
\sum_{k=1}^{n}\Bigl(\sum_{j=1}^{n}\frac{a_{\lambda ij} 
b_{\lambda jk} }{pq}\Bigr)\Delta_{k}.$$
Since every $\Delta_{i}$ is 
represented uniquely by the $\mathbb{R}$-linear combination of 
$\Delta_{1}, \, \cdots, \Delta_{n}$, where the sum of coefficients is equal to one, 
we have 
$\sum_{j=1}^{n}(a_{\lambda ij} b_{\lambda jk}/pq) =\delta_{ik}$, where 
$\delta_{ik}$ is Kronecker delta. 

Pick $(m_{1}, \, \cdots, m_{n}) \in 
(\mathbb{Z}_{\geq 0})^{n}$ such that $m=\sum_{i=1}^{n}m_{i}>0$. 
Then there  exists a $\lambda' \in \Lambda$ 
such that $\sum_{i=1}^{n}(m_{i}/m)\Delta_{i} \in \Sigma_{\lambda'}$. 
Then 
$$\sum_{i=1}^{n}pqm_{i}\Delta_{i}=mpq\sum_{i=1}^{n}(m_{i}/m)\Delta_{i}$$
is uniquely represented by the $\mathbb{R}_{\geq 0}$-linear combination of 
$\Psi_{\lambda' 1}, \, \cdots, \Psi_{\lambda' n}$ where the sum of the coefficients is equal to $mpq$. 
On the other hand, 
$$\sum_{i=1}^{n}pqm_{i}\Delta_{i}
=\sum_{i=1}^{n}pqm_{i}\sum_{j=1}^{n}\frac{a_{\lambda' ij}}{p}\Psi_{\lambda' j}=
\sum_{j=1}^{n}\Bigl(\sum_{i=1}^{n}qa_{\lambda' ij}m_{i} \Bigr)\Psi_{\lambda' j}$$
and $\sum_{i=1}^{n}qa_{\lambda' ij}m_{i} \in \mathbb{Z}$. 
Therefore $\sum_{i=1}^{n}qa_{\lambda' ij}m_{i} \in \mathbb{Z}_{\geq 0}$ and so
\begin{equation*}
\begin{split}
&H^{0}(X, \mathcal{O}_{X} (\lfloor \sum_{i=1}^{n}m_{i}pq(K_{X}+\Delta_{i}) \rfloor))\\
&=H^{0}(X, \mathcal{O}_{X} (\lfloor \sum_{j=1}^{n}q(\sum_{i=1}^{n}
a_{\lambda' ij}m_{i})(K_{X}+\Psi_{\lambda' j})\rfloor))
\end{split}
\end{equation*}
can be identified with the $A$-module of homogeneous 
elements of certain degree in $R^{\lambda'}_{(q)}$. 
Let $\phi_{\boldsymbol{m}\lambda'}:H^{0}(X, \mathcal{O}_{X} 
(\lfloor \sum_{i=1}^{n}m_{i}pq(K_{X}+\Delta_{i}) \rfloor)) \to R^{\lambda'}_{(q)}$, 
where $\boldsymbol{m}=(m_{1}, \, \cdots , m_{n})$, be 
the natural morphism of $A$-modules. 
Similarly, for any $(m'_{1}, \, \cdots, m'_{n}) \in (\mathbb{Z}_{\geq 0})^{n}$, since 
$$\sum_{i=1}^{n}m'_{i}q\Psi_{\lambda' i}
=\sum_{j=1}^{n}(\sum_{i=1}^{n}b_{\lambda' ij} m'_{i})\Delta_{j},$$ 
where $\sum_{i=1}^{n}b_{\lambda' ij} m'_{i} \in \mathbb{Z}_{\geq 0}$, 
we get the natural ring homomorphism $\tau_{\lambda'}: R^{\lambda'}_{(q)} \to R$. 
By the definition of $\phi_{\boldsymbol{m}\lambda'}$ 
and $\tau_{\lambda'}$, for 
any $f \in H^{0}(X, \mathcal{O}_{X} (\lfloor \sum_{i=1}^{n}m_{i}pq(K_{X}+\Delta_{i}) \rfloor))$, 
$\tau_{\lambda'} \circ \phi_{\boldsymbol{m}\lambda'} (f)=f$. 

By the hypothesis, $R^{\lambda'}$ is a finitely generated $A$-algebra. 
Then $R_{(q)}^{\lambda'}$ is also a finitely generated $A$-algebra by Lemma A \ref{cora2}.
Let $g_{\lambda' 1}, \, \cdots , 
g_{\lambda' k_{\lambda'}}$ be the generator of $R^{\lambda'}_{(q)}$. 
Then there exists an $A$-polynomial $F \in A[X_{1}, \, \cdots ,X_{k_{\lambda'}}]$ such that 
$\phi_{\boldsymbol{m}\lambda'} (f)=F(g_{\lambda' 1}, \, \cdots, g_{\lambda' k_{\lambda'}})$. 
Then we have $f=F(\tau_{\lambda'}(g_{\lambda' 1}), \,
\cdots, \tau_{\lambda'} (g_{\lambda' k_{\lambda'}}))$ and so 
$R_{(pq)}$ is generated by 
$\tau_{\lambda'}(g_{\lambda' l})$, where $\lambda' \in \Lambda$ and 
$1 \leq l \leq k_{\lambda'}$. 
Then the lemma follows from Lemma A \ref{lema1}.
\end{proof}

\begin{lem}\label{lem4.3}
To prove Theorem \ref{main-thm}, we may assume that each $K_{X}+\Delta_{i}$ is semi-ample over $U$.
\end{lem}
\begin{proof}
Recall that $\mathcal{C}$ is the rational polytope spanned by $\Delta_{1}, \, \cdots , \Delta_{n}$ 
and $V$ is the finite dimensional affine subspace of ${\rm WDiv}_{\mathbb{R}}(X)$ spanned by $\mathcal{C}$ 
such that ${\rm dim}V=n-1$.
By Lemma \ref{lem3.4}, 
there are finitely many proper birational morphisms $f_{i}: X \to Y_{i}$ over $U$ and 
finitely many rational polytopes $W_{i}$ 
such that $\mathcal{C} \cap \mathcal{E}_{U}(V)= \cup_{i}W_{i}$ and 
if $\Delta \in W_{i}$, then $f_{i}$ is a weak 
log canonical model of $(X, \Delta)$ over $U$.
By applying Lemma B \ref{lemb2} to 
$\mathcal{C}$, $\mathcal{C} \cap \mathcal{E}_{U}(V)$ and $W_{i}$, 
there is a finite $(n-1)$-dimensional rational simplex covering 
$\{ \Sigma_{\lambda} \}_{\lambda \in \Lambda}$ of $\mathcal{C}$ such that 
$\mathcal{C}= \cup_{\lambda}\Sigma_{\lambda}$ and for any $\lambda \in \Lambda$, 
$\mathcal{E}_{U}(V) \cap \Sigma_{\lambda}$ is a face 
of $\Sigma_{\lambda}$ and is contained in 
$W_{j}$ for some $j$. 
By Lemma \ref{lem4.1} it is sufficient to prove the case where 
$\Delta_{1}, \, \cdots, \Delta_{n}$ are vertices of $\Sigma_{\lambda}$ for some $\lambda$. 
By changing indices if necessary, we may assume that 
$\Delta_{1}, \, \cdots, \Delta_{k} \in \mathcal{E}_{U}(V)$ 
and $\Delta_{k+1}, \, \cdots, \Delta_{n} \notin \mathcal{E}_{U}(V)$. 
We note that if $m_{i}\geq 0$ for every $1\leq i \leq n$, then 
$H^{0}(X, \mathcal{O}_{X} ( \lfloor \sum_{i=1}^{n}m_{i}(K_{X}+\Delta_{i}) \rfloor)) \neq0$ 
implies that $m_{k+1}= \cdots =m_{n}=0$. 
Indeed, in this case we see that  
$m_{i}$ is zero for all $i$ or $\sum_{i=1}^{n}(m_{i}/m)\Delta_{i}$ 
is in $\mathcal{E}_{U}(V)$ when $m=\sum_{i=1}^{n}m_{i}>0$.
Therefore we may assume that $\Delta_{i} \in \mathcal{E}_{U}(V)$ for any $1 \leq i \leq n$. 
Then we can find $j$ such that $\Delta_{i} \in W_{j}$ for any $1 \leq i \leq n$. 
Pick a positive integer $d$ such that 
$d(K_{X}+\Delta_{i})$ and $d(K_{Y_{j}}+f_{j*}\Delta_{i})$ are both Cartier for any $1 \leq i \leq n$. 
Then
$$H^{0}(X, \mathcal{O}_{X} 
(\sum_{i=1}^{n}m_{i}d(K_{X}+\Delta_{i}))) 
\cong H^{0}(Y_{j}, \mathcal{O}_{Y_{j}}(\sum_{i=1}^{n}m_{i}d(K_{Y_{j}}+f_{j*}\Delta_{i})))$$ 
and by Lemma A \ref{cora2}, it is sufficient to prove that 
$$\underset{(m_{1}, \, \cdots , \, m_{n}) \in (\mathbb{Z}_{ \geq 0})^{n}}{\bigoplus}
H^{0}(Y_{j}, \mathcal{O}_{Y_{j}} ( \lfloor\sum_{i=1}^{n}m_{i}(K_{Y_{j}}+f_{j*}\Delta_{i}) \rfloor ))$$ 
is a finitely generated $H^{0}(Y_{j}, \mathcal{O}_{Y_{j}})$-algebra.  
Since $f_{j}:X \to Y_{j}$ is a weak log canonical model of $(X,\Delta_{i})$ over $U$ for any $i$, 
by replacing 
$X$ with $Y_{j}$ and 
$K_{X}+\Delta_{i}$ with 
$K_{Y_{j}}+f_{j*}\Delta_{i}$ respectively, we may assume that 
$K_{X}+\Delta_{i}$ is nef over $U$ for any $1 \leq i \leq n$.
Then by \cite[Theorem 8.1]{fujino} and \cite[Theorem 6.7]{tanaka}, 
$K_X+\Delta_i$
is semi-ample over $U$ for any $1\leq i \leq n$. 
\end{proof}

\section{Proof of the main theorem and corollary}\label{sec5}

Now we complete the proof of  
Theorem \ref{main-thm} and Corollary \ref{cor1.2}. 

\begin{proof}[Proof of Theorem \ref{main-thm}]
It is sufficient to prove that
$$\mathcal{R}(\pi, \Delta^{\bullet})
=\underset{(m_{1}, \, \cdots , \, m_{n}) \in (\mathbb{Z}_{ \geq 0})^{n}}{\bigoplus}
H^{0}(X, \mathcal{O}_{X} ( \lfloor \sum_{i=1}^{n}m_{i}D_{i} \rfloor ))$$ 
is a finitely generated $H^{0}(X, \mathcal{O}_{X})$-algebra, 
where $D_{i}=K_{X}+\Delta_{i}$ is semi-ample over $U$ for any $1\leq i\leq n$ 
and $U$ is an affine variety. 
This follows immediately from the following lemma. 
\end{proof}

\begin{lem}\label{lem5.1}
Let $\pi: X \to {\rm Spec}\,A$ 
be a proper morphism from a normal variety to an affine variety and 
let $D_{1}, \, \cdots, D_{n}$ be $\pi$-semi-ample $\mathbb{Q}$-Cartier $\mathbb{Q}$-divisors on $X$.  
Then
$$R=\underset{(m_{1}, \, \cdots, \; m_{n}) \in(\mathbb{Z}_{\geq 0})^{n}}{\bigoplus} 
H^{0}(X, \mathcal{O}_{X}(\lfloor \sum_{i=1}^{n}m_{i}D_{i} \rfloor ))$$
is a finitely generated $A$-algebra. 
\end{lem}

\begin{proof}
By Lemma A \ref{cora2}, 
we may assume that $D_{i}$ is Cartier and free over ${\rm Spec}\,A$ for any $1\leq i \leq n$. 
Thus we may assume that $D_{i}$ is base point free for any $1 \leq i \leq n$. 
Then there is a surjective 
morphism $\overset{r_{i}}{\oplus} \mathcal{O}_{X} \to \mathcal{O}_{X}(D_{i})$ for some 
positive integer $r_{i}$ for every $i$. 
Suppose that $n \geq 2$. 
Set $\mathcal{E}=\mathcal{O}_{X}(D_{1}) 
\oplus \cdots \oplus \mathcal{O}_{X}(D_{n})$. 
Then there is a surjective morphism 
$\overset{r}{\oplus} \mathcal{O}_{X} \to \mathcal{E}$, where 
$r= \sum_{i=1}^{n}r_{i}$. 
Let $p: \boldsymbol{P}_{X}(\mathcal{E}) \to X$ 
be the projective bundle over $X$ associated to $\mathcal{E}$. 
Then there is a natural surjective 
morphism $p^{*}\mathcal{E} \to \mathcal{O}_{\boldsymbol{P}_{X}(\mathcal{E})}(1)$. 
Therefore there exists a surjective morphism 
$\overset{r}{\oplus} \mathcal{O}_{\boldsymbol{P}_{X}(\mathcal{E})} 
\to \mathcal{O}_{\boldsymbol{P}_{X}(\mathcal{E})}(1)$ and so 
$\mid$$\mathcal{O}_{\boldsymbol{P}_{X}(\mathcal{E})}(1)$$\mid$ is base point free.  
Moreover, there is a canonical 
isomorphism of the graded $\mathcal{O}_{X}$-algebras 
$$\underset{l \in \mathbb{Z}_{\geq 0}}{\bigoplus}p_{*} 
\mathcal{O}_{\boldsymbol{P}_{X}(\mathcal{E})}(l) 
\cong \underset{l\in \mathbb{Z}_{\geq 0}}{\bigoplus} 
\left( \underset{m_{1}+\cdots+m_{n}=l}{\bigoplus}
\mathcal{O}_{X}(\sum_{i=1}^{n}m_{i}D_{i}) \right).$$
Therefore
$$R \cong \underset{l \in \mathbb{Z}_{\geq 0}}{\bigoplus} H^{0}
(\boldsymbol{P}_{X}(\mathcal{E}), \mathcal{O}_{\boldsymbol{P}_{X}(\mathcal{E})}(l))$$
as $A$-algebras. 
So we may assume that $n=1$ and 
then the lemma is clear. 
\end{proof}

Let us prove Corollary \ref{cor1.2}. 

\begin{proof}[Proof of Corollary \ref{cor1.2}]
Let $f:Y \to X$ be a resolution of $X$. 
Then we may write 
$$K_{Y}+\Gamma_{i} = f^{*}(K_{X}+\Delta_{i})+E_{i},$$
where $\Gamma_{i} \geq 0$ and 
$E_{i} \geq 0$ have no common components, $f_{*}\Gamma_{i}=\Delta_{i}$ and 
$E_{i}$ is $f$-exceptional. 
Then $\Gamma_{i}$ is a $\mathbb{Q}$-divisor for any $i$ and 
$$\mathcal{R}(\pi, \Delta^{\bullet}) \cong \mathcal{R}(\pi \circ f, \Gamma^{\bullet})$$
as graded $\mathcal{O}_{U}$-algebras.
Moreover, by the definition of 
log canonical pairs, 
$\Gamma_{i}$ is a boundary $\mathbb{Q}$-divisor for every $i$.
Therefore we can reduce it to Theorem \ref{main-thm}.
\end{proof}

\section{Appendix}\label{sec6}

In this section we collect some basic results for the reader's convenience. 

\subsection*{Appendix A. Graded Ring}

In this part, let $k$ be an algebraically 
closed field and let $A$ be a finitely generated $k$-algebra 
such that $A$ is an integral domain. 
Let 
$$R=\underset{(a_{1}, \, \cdots , \, a_{n}) 
\in (\mathbb{Z}_{\geq 0})^{n}}{\bigoplus} R_{(a_{1}, \, \cdots , \, a_{n})}$$
be a graded $A$-algebra such that $R$ is an
integral domain with $R_{(0 , \, \cdots , \,0)}=A$ and let 
$R_{(d)}$ be the $d$-th truncation of $R$. 
More precisely, 
$$R_{(d)}=\bigoplus_{(a_{1}, \, \cdots , \, a_{n}) 
\in (\mathbb{Z}_{\geq 0})^{n}}R_{(da_{1}, \, \cdots, \, da_{n})}. $$ 
For any $\boldsymbol{a} \in (\mathbb{Z}_{\geq 0})^{n}$, we identify 
$R_{\boldsymbol{a}}$ with a homogeneous part of $R$ by the natural inclusion 
$R_{\boldsymbol{a}} \hookrightarrow R$ of $A$-modules. 
For any $f \in R_{\boldsymbol{a}}$, 
we define the $degree$ of $f$ as ${\rm deg}(f)=\boldsymbol{a}$. 

We introduce two well known results. 
For the proofs, see \cite[Proposition 1.2.2]{adhl} and \cite[Proposition 1.2.4]{adhl}.

\begin{lema}\label{lema1}
Suppose that there are $x_{1}, \, \cdots , x_{r} \in R$ 
such that $R_{(d)}$ is a $A$-subalgebra of $R'=A[x_{1}, \, \cdots , x_{r}]$ 
for some positive integer $d$. 
Then $R$ is a finitely generated $A$-algebra.
\end{lema}

\begin{lema}[cf.~{\cite[Corollary 1.2.5]{adhl}}]\label{cora2}
For any $$\boldsymbol{d}=(d_{1}, \, \cdots , d_{n}) 
\in (\mathbb{Z}_{> 0})^{n},$$ the graded 
ring 
$$R_{[\boldsymbol{d}]}=
\underset{(a_{1}, \, \cdots , \, a_{n}) 
\in (\mathbb{Z}_{\geq 0})^{n}}{\bigoplus} R_{(d_{1}a_{1}, \, \cdots ,\, d_{n}a_{n})}$$
is a finitely generated $A$-algebra if and only if $R$ is a finitely generated $A$-algebra.
\end{lema}

We close this part with the following technical result.

\begin{lema}\label{lema3}
Let $\boldsymbol{e}_{1} , \,
\cdots , \boldsymbol{e}_{n} $ be the canonical basis of $\mathbb{Z}^{n}$. 
We set 
$\boldsymbol{e}=\sum_{i=1}^{m} 
\boldsymbol{e}_{j_{i}}$, where $j_{i} \neq j_{i'}$ for $i \neq i'$. 
Then the following conditions are equivalent:
\begin{itemize}
\item[(i)]
The ring 
$$\overline{R}=\underset{(a_{1}, \,
\cdots , \, a_{n}, \, b) \in (\mathbb{Z}_{\geq 0})^{n+1}}{\bigoplus} 
R_{(a_{1}, \, \cdots , \, a_{n})+b \boldsymbol{e}}$$ 
is a finitely generated $A$-algebra.
\item[(ii)]
For any $i$ such that $1 \leq i \leq m$, the ring 
$$\overline{R}^{i}=\underset{(a_{1}, \, \cdots , \, a_{j_{i}-1},\, a_{j_{i}+1},\, \cdots , \, a_{n},\, b) 
\in (\mathbb{Z}_{\geq 0})^{n}}{\bigoplus} 
R_{(a_{1}, \, \cdots , \, a_{j_{i}-1},\, 0,\, a_{j_{i}+1}, \, \cdots , \, a_{n}) +b\boldsymbol{e}}$$
 is a finitely generated $A$-algebra.
\end{itemize}
\end{lema}

\begin{proof}
For any homogeneous element 
$f \in R_{(a_{1}, \, \cdots , \, a_{j_{i}-1},\, 0,\, a_{j_{i}+1}, \, 
\cdots , \, a_{n}) +b \boldsymbol{e}}$, 
we put ${\rm deg}_{i} (f)= (a_{1}, \, \cdots ,\,a_{j_{i}-1},\, a_{j_{i}+1}, \, \cdots , \, a_{n}, \,b) 
\in (\mathbb{Z}_{\geq 0})^{n}$ and we will call it the degree of $f$ in $\overline{R}^{i}$.
Similarly, for any homogeneous element $f \in R_{(a_{1}, \, 
\cdots ,\,a_{n})+b\boldsymbol{e}}$ 
we put ${\rm Deg} (f)= (a_{1}, \, \cdots , \, a_{n}, \, b) \in 
(\mathbb{Z}_{\geq 0})^{n+1}$ 
and we will call it the degree of $f$ in $\overline{R}$. 
Note that $\overline{R}^{i}$ can be identified with the $A$-subalgebra of 
$\overline{R}$ generated by all 
homogeneous elements of $\overline{R}$ whose 
$j_{i}$-th component is zero.  
If $\overline{R}$ is generated as an $A$-algebra by finitely many 
elements of $\overline{R}$, 
which can be assumed to be homogeneous elements, 
then $\overline{R}^{i}$ is generated as an $A$-algebra by the 
generator of $\overline{R}$ whose $j_{i}$-th component  is zero. 
Thus $\overline{R}^{i}$ is also a finitely generated $A$-algebra.

Conversely, suppose that $\overline{R}^{i}$ is 
a finitely generated $A$-algebra for any $1 \leq i \leq m$. 
For simplicity, suppose that $j_{i}=i$ for any $1 \leq i \leq m$. 
Let $g_{i1}, \, \cdots , g_{ir(i)}$ be a 
generator of $\overline{R}^{i}$, where $g_{ij}$ is homogeneous for any 
$1 \leq j \leq r(i)$, and let 
$$\boldsymbol{a}^{ij}=(a_{1}^{ij} ,\, \cdots , \,
a_{(i-1)}^{ij},\, a_{(i+1)}^{ij} , \, \cdots , \, a_{n}^{ij} ,\, b^{ij})$$
be the degree of $g_{ij}$ in $\overline{R}^{i}$. 
Then 
$$g_{ij} \in R_{(a_{1}^{ij}+b^{ij} ,\, \cdots ,\, a_{(i-1)}^{ij}+b^{ij},\, b^{ij},\, 
a_{(i+1)}^{ij}+b^{ij} ,\, \cdots ,\, a_{m}^{ij}+b^{ij},\, a_{(m+1)}^{ij},\, \cdots ,\, a_{n}^{ij})} \:.$$
For any $0 \leq s \leq b^{ij}$, let $g_{ij}(s)$ be the element of $\overline{R}$ such that 
$$g_{ij}(s)=g_{ij} \Bigl(\in R_{(a_{1}^{ij}+b^{ij} ,\, \cdots ,\, 
a_{(i-1)}^{ij}+b^{ij},\, b^{ij}, \, a_{(i+1)}^{ij}+b^{ij} ,\,
\cdots ,\, a_{m}^{ij}+b^{ij},\, a_{(m+1)}^{ij},\, \cdots ,\, a_{n}^{ij})} \Bigr)\;,$$ 
and
\begin{equation*}
\begin{split}
{\rm Deg}(g_{ij}(s))&=(a_{1}^{ij}+b^{ij}-s , \, \cdots , \, 
a_{(i-1)}^{ij}+b^{ij}-s,\, b^{ij}-s,\\
& \qquad a_{(i+1)}^{ij}+b^{ij}-s ,\, \cdots ,\, a_{m}^{ij}+b^{ij}-s,\, 
a_{(m+1)}^{ij},\cdots ,\, a_{n}^{ij},\,s)\;.
\end{split}
\end{equation*}
Then we prove that $\overline{R}$ is generated as an 
$A$-algebra by all $g_{ij}(s)$, where $1\leq i \leq m$, 
$1 \leq j \leq r(i)$ and $0 \leq s \leq b^{ij}$. 

Pick any homogeneous element $f \in \overline{R}$ and let $(a_{1}, \, \cdots ,  a_{n},  b)$ 
be the degree of $f$ in $\overline{R}$. 
Pick an $l$ which satisfies $a_{l}={\rm min}\{ a_{i} \! \mid \! 1 \leq i \leq m\}$. 
Without loss of generality, we may assume that $l=1$. 
Let $f'$ be the element of $\overline{R}^{1}$ such that $f'=f$ as an element of 
$R_{(a_{1}+b, \, \cdots, \, a_{m}+b,\, a_{m+1}, \, \cdots ,\, a_{n})}$ and 
$${\rm deg}_{1}(f')=(a_{2}-a_{1}, \, \cdots ,\, a_{m}-a_{1},\, a_{m+1}, \,
\cdots, \,a_{n},\, a_{1}+b).$$
By the hypothesis there exists a polynomial $F \in A[X_{1}, \,
\cdots , X_{r(1)}]$ such that $f'=F(g_{11}, \, \cdots, g_{1r(1)})$. 
Taking the homogeneous decomposition, we may assume that for any monomial 
$\alpha X_{1}^{t_{1}} \cdots X_{r(1)}^{t_{r(1)}}$ of $F$, where $\alpha \in A\setminus \{0\}$, 
\begin{equation*}
\begin{split}
{\rm deg}_{1}(\alpha g_{11}^{t_{1}} \cdots g_{1r(1)}^{t_{r(1)}})&=
\Biggl(\sum_{j=1}^{r(1)}t_{j}a_{2}^{1j}, \,\cdots ,\, 
\sum_{j=1}^{r(1)}t_{j}a_{n}^{1j},\, \sum_{j=1}^{r(1)}t_{j}b^{1j} \Biggr)\\
&=(a_{2}-a_{1}, \, \cdots ,\, a_{m}-a_{1},\, a_{m+1},\, \cdots, \,a_{n},\, a_{1}+b)\\
& ={\rm deg}_{1}(f')\;. 
\end{split}
\end{equation*} 
Then $\sum_{j=1}^{r(1)}t_{j}b^{1j}=a_{1}+b \geq b$. 
Therefore, for each $1 \leq j \leq r(1)$ and $1 \leq \lambda \leq t_{j}$, 
we may find $0\leq s_{j \lambda} \leq b^{1j}$ 
such that $\sum _{j=1}^{r(1)}\sum_{\lambda=1}^{t_{j}}s_{j \lambda}=b$. 
Then
\begin{equation*}
\begin{split}
&{\rm Deg}(\alpha \Pi_{j=1}^{r(1)}\Pi_{\lambda =1}^{t_{j}} g_{1j}(s_{j \lambda}))\\
&=\Biggl(\sum _{j=1}^{r(1)}\sum_{\lambda=1}^{t_{j}}(b^{1j}-s_{j \lambda}), \,
\sum _{j=1}^{r(1)}\sum_{\lambda=1}^{t_{j}}(a_{2}^{1j}+b^{1j}-s_{j \lambda}), \,\cdots, \\
& \quad \quad \quad \sum _{j=1}^{r(1)}
\sum_{\lambda=1}^{t_{j}}(a_{m}^{1j}+b^{1j}-s_{j \lambda}),\,
\sum_{j=1}^{r(1)}t_{j}a_{m+1}^{1j}\,
,\cdots ,\, \sum_{j=1}^{r(1)}t_{j}a_{n}^{1j},\\
& \quad \quad \quad \quad \sum _{j=1}^{r(1)}
\sum_{\lambda=1}^{t_{j}}s_{j \lambda} \Biggr) \\
&=(a_{1}, \, \cdots, \, a_{n},\, b).
\end{split}
\end{equation*}
Considering all monomials of $F$, $f$ is expressed 
as a polynomial of $g_{ij}(s)$ with coefficients $A$. Thus, 
$\overline{R}$ is generated by all $g_{ij}(s)$ 
as an $A$-algebra. 
\end{proof}

\subsection*{Appendix B. Rational Polytope}
In this part, we collect the definition and some basic properties of rational polytopes. 

\begin{defnb}[Rational polytopes]\label{defb6.1}
Let $\mathcal{C}$ be a subset in a finite dimensional $\mathbb{R}$-vector space. 
Then $\mathcal{C}$ is a {\it polytope} if 
$\mathcal{C}$ is compact and the intersection of finitely many half-spaces, 
or equivalently, the convex hull of finitely many points. 
$\mathcal{C}$ is a {\it rational polytope} if $\mathcal{C}$ is a polytope defined by rational 
half-spaces, or the convex hull of finitely many rational points. 
$\mathcal{F} \subset \mathcal{C}$ is a {\it face} if whenever 
$\sum_{i=1}^{k} r_{i} \boldsymbol{v}_{i} \in \mathcal{F}$, 
where $r_{1}, \cdots , r_{k}$ are positive real numbers 
such that $\sum_{i=1}^{k} r_{i}=1$ 
and $\boldsymbol{v}_{1}, \cdots , \boldsymbol{v}_{k}$ belong to $\mathcal{C}$, then 
$\boldsymbol{v}_{1}, \, \cdots , \boldsymbol{v}_{k}$ belong to $\mathcal{F}$. 
We call $\mathcal{F}$ {\it proper face} of $\mathcal{C}$ if $\mathcal{F} \subsetneq \mathcal{C}$.
\end{defnb}

By the definition, a face of a 
polytope (resp.~rational polytope) is also a polytope (resp.~rational polytope).

\begin{lemb}\label{lemb1}
Let $\mathcal{C}$ be a rational 
polytope in $\mathbb{R}^{n}$ such that ${\rm dim}\,\mathcal{C}=n$ 
and let $p_{1} , \, \cdots , p_{m}$ be its vertices. 
For every $1 \leq i \leq m$, let $\mathcal{C}_{i}$ be the rational polytope spanned by 
$p_{1}, \, \cdots, p_{i-1},p_{i+1},\, \cdots, p_{m}$.
If $m>n+1$, then there exists a 
rational point $p$ in $\mathcal{C}$ such that $p \in \cap_{i=1}^{m} \mathcal{C}_{i}$. 
\end{lemb}

\begin{proof}
We prove this lemma in several steps.

\setcounter{step}{0}
\begin{step}
In this step we reduce the lemma to the case where $m=n+2$ by the induction on $m$.

Suppose that the statement is true in the case of $m$ vertices. 
Then there 
exists a rational point $p$ in $\mathcal{C}$ such 
that $p \in \cap_{i=1}^{m} \mathcal{C}_{i}$. 
In the case of $m+1$ vertices, by changing indices if necessary, we may assume that 
${\rm dim}\,\mathcal{C}_{m+1}={\rm dim}\,\mathcal{C}=n$. 
Let $\mathcal{C}_{i}'$ be the rational polytope spanned by  
$p_{1}, \, \cdots, p_{i-1}, p_{i+1},\, \cdots, p_{m} $ for every $1 \leq i \leq m$. 
Then we have $\mathcal{C}'_{i} \subset \mathcal{C}_{i}$ and 
$\mathcal{C}_{i}' \subset \mathcal{C}_{m+1}$ for every $1 \leq i \leq m$. 
By the induction hypothesis, there is a rational point $p$ such that $p \in \cap_{i=1}^{m} \mathcal{C}'_{i}$.
Thus $p \in \cap_{i=1}^{m} \mathcal{C}'_{i} \subset 
\cap_{i=1}^{m+1} \mathcal{C}_{i}$ and so we are done. 
Therefore we may assume that $m=n+2$. 

By an appropriate affine transformation 
and changing the indices if necessary, we may assume that 
$p_{1}, \, \cdots , p_{n}$ are 
canonical basis of $\mathbb{R}^{n}$ and $p_{n+1}$ is the origin. 
Then we may write $p_{n+2}=(a_{1}, \, \cdots ,a_{n})$, where $a_{i} \in \mathbb{Q}$.
\end{step}
\begin{step}
We first prove the case where $a_{i} \geq 0$ for any $1 \leq i \leq n$. 
In this case, we have $\sum_{i=1}^{n}a_{i}>1$. 
Indeed, if 
$\sum_{i=1}^{n}a_{i}\leq1$, then 
$p_{n+2}=\sum_{i=1}^{n} a_{i}p_{i}+(1-\sum_{i=1}^{n}a_{i})p_{n+1}$. 
This contradicts to the hypothesis that 
$p_{1} ,\, \cdots, p_{n+2}$ are the vertices of $\mathcal{C}$. 
Set $a=\sum_{i=1}^{n}a_{i}$ and let $p$ be $(a_{1}/a, \, \cdots , a_{n}/a)$. 
Then $p$ is a rational point in $\mathcal{C}$ and 
\begin{equation*}
\begin{split}
p&=\sum_{i=1}^{n}\frac{a_{i}}{a} p_{i} \;(\in \mathcal{C}_{i}\; {\rm for} \;i=n+1,\;n+2)\\
&=\Bigl(\frac{1}{a} \Bigr)p_{n+2}+\Bigl(1-\frac{1}{a}\Bigr)p_{n+1} 
\;(\in \mathcal{C}_{i} \;{\rm for\; any}\; 1\leq i \leq n ).
\end{split}
\end{equation*}
Thus $p \in \cap_{i=1}^{n+2} \mathcal{C}_{i}$.
\end{step}

\begin{step}
We prove the case where $a_{i} < 0$ for some $i$. 
 By changing indices if necessary, we may assume that $a_{1}, \, \cdots , a_{l}<0$ 
and $a_{l+1}, \, \cdots , a_{n} \geq 0$ for some $l$. 
Set $a= - \sum_{i=1}^{l}a_{i}$ and $b=\sum_{i=l+1}^{n}a_{i}$. 

 If $1+a \leq b$, then let $p$ be $(0, \, \cdots, 0, a_{l+1}/b, \, \cdots, a_{n}/b)$. 
Then $p$ is a rational point 
in $\mathcal{C}$ and 
\begin{equation*}
\begin{split}
p&=\sum_{i=l+1}^{n}\Bigl( \frac{a_{i}}{b} \Bigr)p_{i} \;(\in 
\mathcal{C}_{i}\; {\rm for} \;i=1,\, \cdots , l,  n+1, n+2)\\
&=\frac{1}{b}  \cdot \Bigl( -\sum_{i=1}^{l} a_{i}p_{i} +p_{n+2}\Bigr)+
\Bigl( 1-\frac{1+a}{b}\Bigr)p_{n+1} \\
& \qquad \qquad (\in \mathcal{C}_{i} \;{\rm for\; any}\; l+1\leq i \leq n ).
\end{split}
\end{equation*}

If $1+a > b$, then set $p=(0, \, \cdots,0, a_{l+1}/(1+a),\, \cdots, a_{n}/(1+a)).$
Then $p$ is a rational point 
in $\mathcal{C}$ and 
\begin{equation*}
\begin{split}
p&=\sum_{i=l+1}^{n}\Bigl( \frac{a_{i}}{1+a} \Bigr)p_{i} +\Bigl( 1-\frac{b}{1+a}\Bigr)p_{n+1}\\
&\qquad \qquad\qquad(\in \mathcal{C}_{i}\; {\rm for} \;i=1,\, \cdots , l, n+2)\\
&=\frac{1}{1+a}\Bigl( -\sum_{i=1}^{l} a_{i}p_{i} +p_{n+2}\Bigr)\;
 (\in \mathcal{C}_{i} \;{\rm for\; any}\; l+1\leq i \leq n+1 ).
\end{split}
\end{equation*}
Thus, $p \in \cap_{i=1}^{n+2} \mathcal{C}_{i}$.
\end{step}
\end{proof}

\begin{lemb}\label{lemb2}
Let $\mathcal{C}$, $\mathcal{D}$, 
$\mathcal{D}_{1}, \, \cdots, \mathcal{D}_{r}$ be rational polytopes in 
$\mathbb{R}^{n}$ such that ${\rm dim}\, \mathcal{C}=n$ and 
$\mathcal{C} \supset \mathcal{D} = 
\cup_{i=1}^{r}\mathcal{D}_{i}$. 
Then there exists a finite $n$-dimensional 
rational simplex covering $\{ \Sigma_{\lambda} \}_{\lambda}$ 
of $\mathcal{C}$ such that $\mathcal{C}= 
\cup_{\lambda}\Sigma_{\lambda}$ and for any $\lambda$, 
$\mathcal{D} \cap \Sigma_{\lambda}$ is a face of $\Sigma_{\lambda}$ and contained in 
$\mathcal{D}_{i}$ for some $1\leq i \leq r$.
\end{lemb}

\begin{remb}\label{remb4}
In Lemma B \ref{lemb2},  the word \lq\lq a simplex covering\rq\rq 
\ 
means a covering by simplices. In particular it does not mean a triangulation. 
Similarly, the covering $\mathcal D=\cup _{i=1}^r \mathcal D_i$ of $\mathcal D$ 
by $\{\mathcal D_i\}_{i=1}^r$  
need not be a subdivision of $\mathcal D$ by $\{\mathcal D_i\}_{i=1}^r$. 
\end{remb}

\begin{proof}[Proof of Lemma B \ref{lemb2}]
We prove it by the induction on the dimension of $\mathcal{C}$. 
In the case  of ${\rm dim}\,\mathcal{C} =0$, the statement is trivial. 
So we may assume that ${\rm dim}\,\mathcal{C}>0$. 

Since $\mathcal{D}$ is 
a rational polytope, there are finitely many affine functions $H_{1}, \,\cdots , H_{k}$ 
such that $\mathcal{D}$ is the intersection of these half spaces 
$(H_{j})_{\geq0}=\{ \boldsymbol{x} \in \mathbb{R}^{n}  \mid H_{j}(\boldsymbol{x}) \geq 0\}$.
Set $(H_{j})_{\leq 0}=\{ \boldsymbol{x} \in \mathbb{R}^{n}  \mid H_{j}(\boldsymbol{x}) \leq 0\}$ 
and consider $\mathcal{C}'_{j}=(H_{j})_{\leq 0} \cap \mathcal{C}$. 
Note that if ${\rm dim}\,\mathcal{C}'_{j}<n$ for some $j$, then 
$\mathcal{C}'_{j}$ is a proper face of $\mathcal{C}$.
Therefore, if we pick all indices $j$ satisfying the condition that ${\rm dim}\,\mathcal{C}'_{j}=n$, 
then we have $\mathcal{C}=(\cup_{j}\mathcal{C}'_{j})\cup \mathcal{D}$. 

Pick an index $j$ such that ${\rm dim}\,\mathcal{C}'_{j}=n$. 
We note that $\mathcal{C}'_{j} \cap \mathcal{D}$ is 
a rational polytope contained in a proper  face of $\mathcal{C}'_{j}$. 
Fix an interior rational point $p$ of $\mathcal{C}'_{j}$ and let $\mathcal{F}$ be an 
$(n-1)$-dimensional face of $\mathcal{C}'_{j}$.
Then $\mathcal{F} \cap \mathcal{D}$ 
and each $\mathcal{F} \cap \mathcal{D}_{i}$ is empty or 
a rational polytope in $\mathbb{R}^{n-1}$ and 
$\mathcal{F} \supset \mathcal{F} \cap 
\mathcal{D}= \cup_{i=1}^{r}\mathcal{F} \cap \mathcal{D}_{i}$.
Therefore $\mathcal{F}$, $\mathcal{F} \cap \mathcal{D}$ and 
$\mathcal{F} \cap \mathcal{D}_{i}$, where $1\leq i \leq r$, satisfy the 
induction hypothesis. 
So there is a finite $(n-1)$-dimensional rational 
simplex covering $\{ \Sigma'_{\lambda'} \}_{\lambda'}$ 
of $\mathcal{F}$ which satisfies the conditions of the lemma. 
Let $\Sigma''_{\lambda'}$ be the convex hull spanned by $p$ and $\Sigma'_{\lambda'}$. 
Then $\{ \Sigma''_{\lambda'} \}_{\lambda'}$ is 
a finite $n$-dimensional rational simplex covering 
of the convex hull spanned by $p$ and $\mathcal{F}$ satisfying 
the conditions of the lemma. 
Considering all $(n-1)$-dimensional 
faces of $\mathcal{C}'_{j}$, we may find a finite $n$-dimensional 
rational simplex covering of $\mathcal{C}'_{j}$ satisfying 
the conditions of the lemma. 

Considering all $\mathcal{C}'_{j}$ 
and a triangulation of each $\mathcal{D}_{i}$, where 
$1 \leq i \leq r$, we get a desired covering.
\end{proof}


\end{document}